\documentclass{amsart}

\usepackage{amsmath,amssymb,amsthm,pinlabel}

\def\co{\colon\thinspace}

\newcommand{\R}{\mathbb{R}}
\newcommand{\Z}{\mathbb{Z}}

\newcommand{\tb}{{\tt tb}}

\newcommand{\rot}{{\tt rot}}
\newcommand{\xist}{\xi_{\mathrm{st}}}

\newtheorem{thm}{Theorem}
\newtheorem{lem}[thm]{Lemma}
\newtheorem{prop}[thm]{Proposition}
\newtheorem{cor}[thm]{Corollary}

\theoremstyle{definition}

\theoremstyle{remark}
\newtheorem*{rem}{Remark}

\newtheorem*{ack}{Acknowledgement}

\begin{document}

\title[Knots in contact $3$-manifolds]{Neighbourhoods and isotopies of
knots in contact $3$-manifolds}

\author{Fan Ding}
\address{Department of Mathematics, Peking University,
Beijing 100871, P.~R. China}
\email{dingfan@math.pku.edu.cn}
\author{Hansj\"org Geiges}
\address{Mathematisches Institut, Universit\"at zu K\"oln,
Weyertal 86--90, 50931 K\"oln, Germany}
\email{geiges@math.uni-koeln.de}
\date{}


\thanks{F.~D.\ is partially supported by
grant no.\ 10631060 of the National Natural Science Foundation
of China and a DAAD -- K.~C.~Wong fellowship,
grant no.\ A/09/99005, at the Universit\"at zu K\"oln.}

\begin{abstract}
We prove a neighbourhood theorem for arbitrary knots
in contact $3$-manifolds. As an application we show that
two topologically isotopic Legendrian knots in a contact $3$-manifold 
become Legendrian isotopic after suitable stabilisations.
\end{abstract}

\maketitle

\section{Introduction}
For an oriented Legendrian knot $K$ in a $3$-dimensional contact manifold
$(M,\xi )$, i.e.\ a knot everywhere tangent to the
contact structure~$\xi$, there is a well-defined notion of
positive or negative stabilisation. By the Darboux theorem,
$(M,\xi)$ is locally diffeomorphic to $\R^3$ with its standard
contact structure $\xist=\ker (dz+x\, dy)$. In such a neighbourhood,
the Legendrian knot $K$ can be represented by its front projection
to the $yz$-plane; the $x$-coordinate can be recovered from the front
as $x=-dz/dy$. In this local picture, the stabilisation $S_{\pm}K$
of $K$ is obtained by adding a zigzag to the front, oriented downwards
(resp.\ upwards) for the positive (resp.\ negative) stabilisation.
Positive and negative stabilisations commute with each other,
and we write $S^m_+S^n_-K$ for an $m$-fold positive and $n$-fold
negative stabilisation of~$K$.
For more background information see \cite{etny05}
and~\cite{geig08}.

The following theorem says that, up to stabilisation,
the classification of Legendrian knots is purely topological.

\begin{thm}
\label{thm:isotopy}
If two oriented Legendrian knots $K_0$ and $K_1$ in a $3$-dimensional
contact manifold $(M,\xi)$ are topologically isotopic,
one can find Legendrian isotopic stabilisations $S_+^{m_0}S_-^{n_0}K_0$
and $S_+^{m_1}S_-^{n_1}K_1$.
\end{thm}

For $(M,\xi)=(\R^3,\xist)$ this theorem was proved by
Fuchs and Tabachnikov~\cite[Theorem~4.4]{futa97}.
Dymara~\cite{dyma04} has suggested to prove the theorem for general
$(M,\xi)$ by reducing it to that special case using the Darboux theorem,
without providing details. In order to make such an argument precise,
one needs considerations very much like those in the proofs of
Lemma~\ref{lem:isotopy} and Theorem~\ref{thm:isotopy} below.

In the present note we give a proof based on convex surface theory
and a neighbourhood theorem for arbitrary knots in contact $3$-manifolds;
this argument does not depend on the result of Fuchs and Tabachnikov.

\begin{rem}
For homologically trivial Legendrian knots one can define
the Thurston--Bennequin invariant $\tb$ and the rotation number~$\rot$
(relative to a choice of Seifert surface).
The parity of the sum $\tb+\rot$ is invariant under stabilisation,
so the theorem implies that this parity is constant within any
(homologically trivial) knot type. See~\cite[Remark~4.6.35]{geig08}
for a more general statement of this parity condition.
\end{rem}

\begin{ack}
We thank Bijan Sahamie for useful comments.
\end{ack}
\section{A neighbourhood theorem}
Denote the obvious coordinates on the
manifold $S^1\times\R^2$ by $\theta ,x,y$.  Throughout this note we
write $S^1$ as shorthand for $S^1\times\{ 0\}\subset S^1\times\R^2$.

\begin{lem}
Let $\alpha$ be a contact form defined near $S^1\subset S^1\times\R^2$.
Then there is a neighbourhood $U$ of $S^1$ and a smooth function
$\lambda\co U\rightarrow\R^+$ with $\lambda|_{S^1}\equiv 1$
and such that along $S^1$ the Reeb vector field of the contact form
$\lambda\alpha|_U$ is transverse to~$S^1$.
\end{lem}

\begin{proof}
We make the ansatz $\lambda (\theta,x,y)=1+ax+by$, where $a,b$ are
real constants that we shall have to choose judiciously.
Write $\beta:=\lambda\alpha$. Then
\[ d\beta = (1+ax+by)\, d\alpha + a\, dx\wedge\alpha +b\, dy\wedge\alpha.\]
We want to choose $a,b\in\R$ such that $i_{\partial_{\theta}}d\beta$
does not vanish along~$S^1$. We define three
smooth functions on $S^1$ by
\[ \lambda_1:=d\alpha(\partial_{\theta},\partial_x)|_{S^1},\;\;\;
\lambda_2:=d\alpha(\partial_{\theta},\partial_y)|_{S^1},\;\;\;
\mu:=\alpha(\partial_{\theta})|_{S^1}.\]
Then
\[ i_{\partial_{\theta}}d\beta|_{S^1}=(\lambda_1(\theta )-a\mu(\theta))\, dx
+(\lambda_2(\theta)-b\mu(\theta))\, dy.\]

Since $\alpha$ is a contact form, we have
\[ (\lambda_1(\theta),\lambda_2(\theta))\neq (0,0) \;\;\;\mbox{\rm on}\;\;\;
\{\theta\in S^1\co \mu(\theta)=0\}.\]
Hence, any point $(a,b)\in\R^2$ not in the image of the map
\[ \begin{array}{ccc}
\{\theta\in S^1\co \mu(\theta)\neq 0\} & \longrightarrow & \R^2\\
\theta                              & \longmapsto     & (\lambda_1(\theta),
                                                         \lambda_2(\theta))
                                                         /\mu(\theta)
\end{array} \]
will satisfy our requirements. By Sard's theorem such points
exist in abundance.
\end{proof}

The following proposition includes as special cases the neighbourhood
theorems for, respectively, Legendrian and transverse knots,
cf.~\cite[Section~2.5]{geig08}.

\begin{prop}
Suppose $\xi_i=\ker\alpha_i$, $i=1,2$, are two positive
contact structures defined near $S^1\subset S^1\times\R^2$ with the
property that there is a smooth function $\mu\co S^1\rightarrow\R^+$
such that $\alpha_1(\partial_{\theta})|_{S^1}=
\mu\alpha_2(\partial_{\theta})|_{S^1}$. Then there is a neighbourhood
$U$ of $S^1$ and a contactomorphism $f\co (U,\xi_1)\rightarrow
(f(U),\xi_2)$ equal to the identity on~$S^1$.
\end{prop}

\begin{proof}
By extending $\mu$ to a smooth positive function on $S^1\times\R^2$
and replacing $\alpha_2$ by $\mu\alpha_2$ we may assume
that $\alpha_1(\partial_{\theta})|_{S^1}=\alpha_2(\partial_{\theta})|_{S^1}$.
Moreover, the lemma allows us to assume that the Reeb vector field
$R_i$ of $\alpha_i$ is transverse to~$S^1$ for $i=1,2$.

Then the vector field
\[ X_i:=\bigl(\partial_{\theta}-\alpha_i(\partial_{\theta})R_i\bigr)|_{S^1} \]
is a non-zero section of $\xi_i|_{S^1}$. Choose a section $Y_i$
of $\xi_i|_{S^1}$ linearly independent of $X_i$ and such that
$R_i,X_i,Y_i$ constitutes a positive frame
of $T(S^1\times\R^2)|_{S^1}$. Then $R_i,\partial_{\theta},Y_i$
is likewise a positive frame of $T(S^1\times\R^2)|_{S^1}$.

We now find a germ of an orientation-preserving diffeomorphism $g$
near $S^1$ with the properties
\begin{itemize}
\item[(i)] $g|_{S^1}=\mbox{\rm id}$,
\item[(ii)] $Tg(R_1)=R_2$ and $Tg(Y_1)=Y_2$ along $S^1$.
\end{itemize}
Then also $Tg(X_1)=X_2$, so
$\alpha_1$ and $\beta_1:=g^*\alpha_2$ are contact forms
near $S^1$ that coincide along~$S^1$. Hence, in a sufficiently
small neighbourhood of~$S^1$, we have a $1$-parameter
family $(1-t)\alpha_1+t\beta_1$ of contact forms; this homotopy
of contact forms is stationary along~$S^1$.
Gray stability \cite[Theorem~2.2.2]{geig08}
gives us a germ of a diffeomorphism $h$
near $S^1$ sending $\ker\alpha_1$ to $\ker\beta_1$
and equal to the identity along~$S^1$.
The composition $f:=g\circ h$ is the desired germ of a diffeomorphism
near~$S^1$.
\end{proof}

\begin{cor}
\label{cor:nbhd}
Any knot $K$ in a $3$-dimensional contact manifold $(M,\xi)$ has
a neighbourhood $U$ such that $\xi|_U$ is tight.
\end{cor}

\begin{proof}
Identify a neighbourhood of $K\subset M$ with a neighbourhood
of $S^1\subset S^1\times\R^2$ such that $K$ becomes identified with~$S^1$.
We continue to write $\xi =\ker\alpha$
for the contact structure in this neighbourhood; the
identification of neighbourhoods may be done in such a way that
$\xi$ is a positive contact structure near $S^1\subset S^1\times\R^2$.
Define a smooth function $\mu\co S^1\rightarrow\R$ by
$\mu=\alpha(\partial_{\theta})|_{S^1}$.

The $1$-form $\alpha_0=dy-x\, d\theta$ defines the standard positive
tight contact structure $\xi_0=\ker\alpha_0$ on $S^1\times\R^2$.
Now consider the embedding $i\co S^1\rightarrow S^1\times\R^2$ given by
$\theta\mapsto (\theta ,-\mu(\theta),0)$. Then
\[ i^*\alpha_0(\partial_{\theta})=\mu
=\alpha(\partial_{\theta})|_{S^1}.\]
By the preceding proposition there is
a neighbourhood of $K$ contactomorphic to a neighbourhood
of $i(S^1)$ in the tight contact manifold $(S^1\times\R^2,\xi_0)$.
\end{proof}
\section{Proof of the isotopy theorem}
We first want to prove a local version of Theorem~\ref{thm:isotopy}
(see Lemma~\ref{lem:isotopy} below).
We begin with one of the two model situations of such a local isotopy.
In $S^1\times\R^2$ with the standard contact structure
$\xi_0=\ker (dy-x\, d\theta)$ we have for each $s\in\R$
a Legendrian knot $\Lambda_s:=S^1\times\{ (0,s)\}$.
In the front projection to the $(\theta,y)$-plane,
where we think of $S^1$ as $\R/2\pi\Z$, the knot
$\Lambda_s$ is represented by a horizontal line
at level $y=s$ (see Figure~\ref{figure:isotopy-Lambda}).
We give $\Lambda_s$ the orientation corresponding to the
positive $\theta$-direction.

The annulus
\[ A_0:=\{ (\theta ,x,y)\in S^1\times\R^2\co x^2+y^2=1,\; x\geq 0\} \]
with boundary $\Lambda_1\sqcup\Lambda_{-1}$ (one of them with
reversed orientation) is transverse to the
contact vector field $x\partial_x+y\partial_y$ and hence a
convex surface in the sense of Giroux~\cite{giro91}.
The dividing set of~$A_0$, i.e.\ the set of points where the contact
vector field is tangential to the contact structure, consists
of a single circle $A_0\cap\{ y=0\}$.

\begin{lem}
\label{lem:model}
The Legendrian knots $S_+\Lambda_1$ and $S_+\Lambda_{-1}$ are
Legendrian isotopic inside any given neighbourhood of
the annulus~$A_0$.
\end{lem}

\begin{rem}
It follows from a result of Traynor~\cite{tray97} that no such isotopy
exists between the unstabilised knots $\Lambda_1$ and~$\Lambda_{-1}$.
\end{rem}

\begin{proof}[Proof of Lemma~\ref{lem:model}]
The $x$-coordinate of a point on a Legendrian knot is given as the
slope $dy/d\theta$ of the front projection at the corresponding point
in the $(\theta ,y)$-plane. Hence, the condition that
a Legendrian knot be close to the annulus~$A_0$ translates into
$y^2+(dy/d\theta)^2$ being close to $1$ for all points
on the front projection of the knot. An isotopy of the front
of $S_+\Lambda_1$ to that of $S_+\Lambda_{-1}$ via fronts
that satisfy this condition is shown in Figure~\ref{figure:isotopy-Lambda}.
\end{proof}

\begin{figure}[h]
\labellist
\small\hair 2pt
\pinlabel $\Lambda_1$ [r] at 0 470
\pinlabel $\Lambda_0$ [r] at 0 397
\pinlabel $\Lambda_{-1}$ [r] at 0 325
\endlabellist
\centering
\includegraphics[scale=0.45]{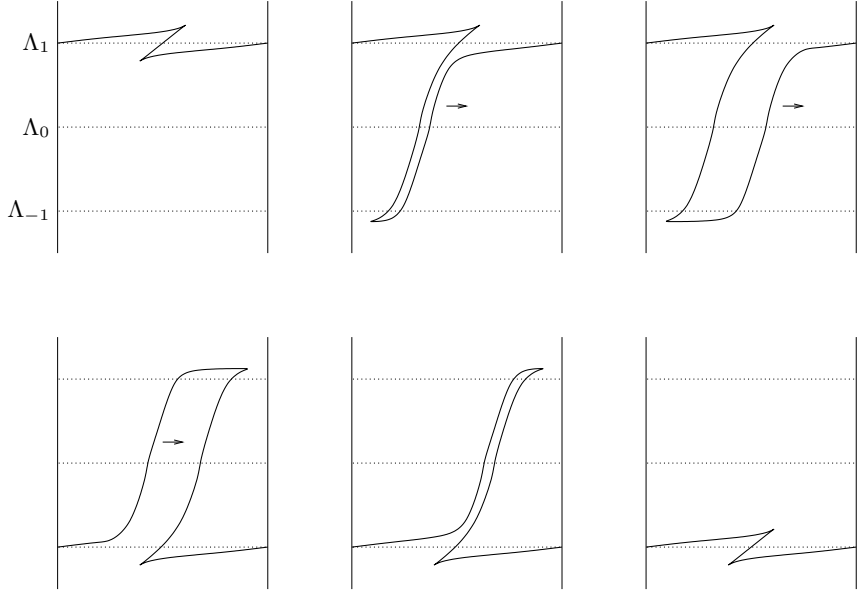}
  \caption{Isotoping $S_+(\Lambda_1)$ to $S_+(\Lambda_{-1})$
near~$A_0$.}
  \label{figure:isotopy-Lambda}
\end{figure}

\begin{lem}
\label{lem:isotopy}
Let $\xi$ be a tight contact structure on $S^1\times\R^2$,
and $K_0,K_1$ oriented Legendrian knots in $(S^1\times\R^2,\xi)$
topologically isotopic to~$S^1$. Then one can find
Legendrian isotopic stabilisations of $K_0$ and~$K_1$.
\end{lem}

\begin{proof}
Write $D_r$ for the open $2$-disc of radius $r$ in $\R^2$, and
$\overline{D}_r$ for its closure. Choose $R>0$ sufficiently
large such that $K_0$ and $K_1$ are topologically
isotopic inside $S^1\times D_R$. Let $K$ be an oriented Legendrian
knot topologically isotopic to $S^1\times\{ (0,3R)\}$
inside $S^1\times (\R^2\setminus \overline{D}_{2R})$.
We claim that suitable stabilisations of $K_0$ and $K_1$ are
Legendrian isotopic to a stabilisation of $K$, and hence
Legendrian isotopic to each other.

The key to proving this claim (for $K_0$, say) is that,
by construction, $K$ and $K_0$ (one of them with reversed
orientation) bound an embedded annulus $A$ in $S^1\times\R^2$.
Beware that $K_0$ and $K_1$ do not, in general, bound
an annulus; an example is given by the
Whitehead link.

Write $t_A(K),t_A(K_0)$ for the twisting of the contact planes
along $K,K_0$, respectively, relative to the framing induced by~$A$.
By stabilising $K$ and $K_0$, if necessary, we may assume that
$t_A(K),t_A(K_0)\leq 0$. Then $A$ can be perturbed (relative to
its boundary $\partial A=K\sqcup K_0$) into a convex surface,
see~\cite[Proposition~3.1]{hond00}. We continue to write $A$ for the
annulus after this and the following perturbations.
If there is a boundary parallel dividing curve on~$A$, then the
corresponding boundary component can be destabilised
without affecting the convexity of~$A$,
see~\cite[Proposition~3.18]{hond00}.

So we may assume that $K$ and $K_0$ are connected by a convex annulus $A$
without boundary parallel dividing curves. The Giroux
criterion~\cite[Proposition~4.8.13]{geig08} tells us that,
since $\xi$ is tight, there are no homotopically trivial
closed curves in the dividing set of~$A$. Thus, the dividing set
consists either of an even number of curves connecting $K$
with~$K_0$, or a collection of simple closed curves parallel
to $K$ and~$K_0$. We now use the Giroux flexibility
theorem~\cite[Proposition~II.3.6]{giro91},
cf.~\cite[Theorem~3.4]{hond00} and~\cite[Theorem~4.8.11]{geig08},
to bring the annulus $A$ into standard form.

In the first case we can perturb $A$ such that
its characteristic foliation is given by curves parallel to $K$
and~$K_0$; this Legendrian ruling of $A$ defines a Legendrian isotopy
between $K$ and~$K_0$.

In the second case, which occurs
if $t_A(K)=t_A(K_0)=0$, we can assume that the characteristic foliation
consists of curves going from $K$ to $K_0$, with $K$ and $K_0$
Legendrian divides (i.e.\ curves in the characteristic foliation
consisting entirely of singular points, where the contact planes coincide
with the tangent planes to~$A$), and one further Legendrian divide between
each pair of dividing curves. Then each of the annuli between two
adjacent Legendrian divides has a characteristic foliation
like our model annulus $A_0$. Since the characteristic foliation
determines the germ of the contact structure near the
surface, cf.~\cite[Theorem~2.5.22]{geig08}, Lemma~\ref{lem:model}
tells us that the stabilised knots $S_+K$ and $S_+K_0$ are Legendrian
isotopic.
\end{proof}

\begin{proof}[Proof of Theorem~\ref{thm:isotopy}]
Let $\phi_t\co S^1\rightarrow M$, $t\in [0,1]$, be
an isotopy of topological embeddings with
$\phi_i(S^1)=K_i$ for $i=0,1$. By Corollary~\ref{cor:nbhd},
for each $t\in [0,1]$ there is a neighbourhood
$U_t$ of $\phi_t(S^1)$, diffeomorphic to
$S^1\times\R^2$ under a diffeomorphism sending
$\phi_t(S^1)$ to~$S^1$, with $\xi|_{U_t}$ tight, and a real number
$\varepsilon_t>0$ such that $\phi_s(S^1)\subset U_t$
for all $s\in (t-\varepsilon_t,t+\varepsilon_t)\cap [0,1]$.

By the Lebesgue lemma on open coverings of compact metric spaces,
there is a positive integer $N$ such that
for each $j\in\{1,\ldots,N\}$ the interval $[(j-1)/N,j/N]$
is contained in $(t_j-\varepsilon_{t_j},t_j+\varepsilon_{t_j})$
for some $t_j\in [0,1]$. We abbreviate $U_{t_j}$ to~$U_j$.
Notice that $\phi_{j/N}(S^1)\subset U_j\cap U_{j+1}$.

Relabel $K_1$ as $K_N$. For $j\in\{ 1,\ldots ,N-1\}$, let $K_j$ be
a Legendrian approximation of $\phi_{j/N}(S^1)$ contained in
the neighbourhood $U_j\cap U_{j+1}$;
such a $C^0$-close Legendrian approximation
exists by~\cite[Theorem~3.3.1]{geig08}.

By the preceding lemma, applied to the Legendrian knots
$K_{j-1}$ and $K_j$ in $U_j$, suitable stabilisations of
$K_{j-1}$ and $K_j$ are Legendrian isotopic, $j=1,\ldots ,N$.
It follows that some stabilisation of $K_0$ is Legendrian isotopic to
some stabilisation of~$K_N$
(which was the $K_1$ in the statement of the theorem).
\end{proof}

\end{document}